\documentclass[11pt]{amsart}
\usepackage{fullpage}
\usepackage{hyperref}
\usepackage{amssymb, amsmath, amsfonts}

\usepackage{enumerate}

\newtheorem{thm}{Theorem}[section]

\newtheorem{prop}[thm]{Proposition}

\numberwithin{equation}{section}

\theoremstyle{definition}

\newtheorem{defin}[thm]{Definition}

\newtheorem{rem}[thm]{Remark}

\newcommand{\bx}{\bar{x}}

\newcommand{\et}{\mathbb{E}}
\newcommand{\ee}{\mathbb{E}}

\makeatletter
\@namedef{subjclassname@2020}{%
  \textup{2020} Mathematics Subject Classification}
\makeatother

\title[Stochastic control on the half-line]{Stochastic exit-time control on the half-line over a finite horizon}

\author[D. Zawisza]{Dariusz Zawisza}

\address{\noindent Dariusz Zawisza, \newline \indent
Jagiellonian University \newline \indent Faculty of Mathematics and Computer Science \newline \indent  Institute of Mathematics \newline \indent    {\L}ojasiewicza  6 \newline \indent 30-348 Krak{\'o}w, Poland }

\email{dariusz.zawisza@uj.edu.pl}

\date{\today}

\begin{document}

\begin{abstract}
We consider a finite-time stochastic drift control problem with the assumption that the control is bounded and the system is controlled until the state process leaves the half-line. Assuming general conditions, it is proved that the resulting parabolic Hamilton--Jacobi--Bellman equation has a classical solution. In fact, we consider an even more general family of semilinear equations, which might be helpful in solving other control or game problems. Not only is the existence result proved, but also a recursive procedure for finding a solution resulting from a fixed-point argument is provided. 
\end{abstract}



\maketitle

\section{Introduction}

The fundamental problem in the theory of exit--time stochastic control is to find an optimal strategy that maximizes (or minimizes) the expected  reward (or cost) accumulated up to the first exit time of the controlled stochastic process from a given domain. We consider a one-dimensional controlled process and define as exit domain the half-line (this is an unbounded domain). In this case, the exit time can be interpreted, for example, as a time of ruin (in actuarial problems), a default (in the credit risk), an extinction time (in the optimal harvesting problem) or finally the ending of the epidemy (in epidemiological studies). In our study, we concentrate on the Hamilton--Jacobi--Bellman (HJB) equation and the Markov feedback policies. There are still not many research papers dedicated to this equation on an unbounded domain.

The most popular method of finding a solution to the HJB equation is viscosity theory (e.g. Fleming and Soner \cite[Chapter V]{FS}). However, it is well-known that the viscosity theory is not useful for determining optimal strategies for the stochastic control functional. Therefore, the main objective of the paper is to prove (under very general conditions) that the associated equation admits a classical solution (class $\mathcal{C}^{2,1}$). The existence of such a solution facilitates many numerical methods and allows one to prove that the Markov feedback policies are indeed optimal. We stress the fact that to prove the main theorem, we use a fixed point type method and present a recursive algorithm, which in the stochastic control framework is usually called the gradient iteration algorithm (it is the continuous-time analogue of the value iteration algorithm), which allows us to prove the uniform convergence of the approximating sequence together with the first derivative on the entire space $(0,+\infty) \times [0,T)$. This is meaningful from the point of view of the feedback strategy convergence because the uniform convergence maintains stability of the algorithm, and consequently is important from the practitioners' point of view, especially in the financial and actuarial industry.
  
More precisely, our paper uses contractive arguments with exponential time-weighted norms (Bielecki's method \cite{Bielecki}), which have recently gained a lot of attention in the stochastic control theory (and its applications). We present a line of reasoning inspired by: Becherer and Schweizer \cite{Becherer}, Bychowska and Leszczy\'{n}ski \cite{Bychowska}, Ph.D. thesis (in Polish) of Zawisza \cite{Zawisza2}, Zawisza \cite{Zawisza1}. In all these papers, a fixed time horizon case (the whole $\mathbb{R}^{n}$ as a domain for the equation) was considered. There are already other papers which address the existence result for semi-linear or HJB equations on entire $\mathbb{R}^{n}$ with the conditions similar to ours (e.g. Addona \cite{A}, Addona et al. \cite{Addona}, Pardoux and Peng \cite{Pardoux}, Rubio \cite{Rubio1}) and even using similar methods. However, we should stress the fact that the extension of the method from the fixed time horizon to the random exit time is not trivial because the estimates of the hitting time   
have to be proved and the boundary behaviour of a diffusion has to be revealed. In addition, our solution is tailor-made for the one-dimensional case only and is not easily extendible to the multidimensional case. 

When it comes to the literature on exit-time stochastic control, it is focused more on exit from a bounded domain and on an infinite-horizon framework and integrates elliptic equations with a BSDE perspective (see, for example, Buckdahn and Nie \cite{Buckdahn} or Royer \cite{Royer} and references therein). We recommend also the books of Fleming and Soner \cite{FS} and Krylov \cite{Krylov} for the classical approach to bounded domain problems. In the unbounded domain parabolic case, we are aware of a recent result of Calvia et al. \cite{Calvia} who solve a stochastic control problem on the half-space of a general Hilbert space. As the primary instrument, they use mild and strong $\mathcal{K}$--solutions. Our result is merely one-dimensional, but at the same time contains topics and cases that are not included in \cite{Calvia}:

\begin{enumerate}
\item First, we show how to handle a one-dimensional model with a non-trivial  inhomogeneous boundary condition and a non-trivial diffusion coefficient, which cannot be done under the framework presented in \cite{Calvia}. Note that in-homogeneous boundary condition is crucial for solving many practical problems (for instance, risk sensitive problems).
  
\item Secondly, we focus on the classical solutions, which is a smaller class of functions in comparison to the class of mild solutions and, to achieve our main result, we introduce a different approach from the one suggested in \cite{Calvia} which allows us to prove the uniform convergence of the recursive sequence on the entire space $(0,+\infty)\times[0,T)$. 
\end{enumerate}

There is also a branch of literature connected with so-called finite fuel problem (e.g. Rokhlin and Mironenko \cite{Rokhlin}) where the framework is sufficiently comprehensive to  encompass the formulation of our problem as well. 
However, we are not aware of any findings concerning classical solutions or the fixed point iteration methods.

The remainder of this paper is structured as follows. In Section 2, we introduce some notational conventions. Section 3 contains a description of the problem.  The main result (Theorem \ref{main}) and its proof are presented in Section 4. Section 5 discusses the approach to dealing with the Lipschitz continuity of the expected terminated hitting time for diffusions, which is crucial for applicability of our main theorem. 

\section{Problem description} 
It is high time to outline the main problem which we are interested in. We use standard notation from SDE theory, together with conventional definitions of partial derivatives, norms, and probabilistic frameworks (see Appendix for more details).
 We consider the HJB equation of the form 
\begin{multline}  \label{first:eq} D_{t}u+ \tfrac{\sigma^{2}(x,t)}{2} D^{2}_{x} u  +\max_{\alpha \in A}  \bigl(b(x,t,\alpha) D_{x}u  +  h(x,t,\alpha) u  + l(x,t,\alpha)\bigr) =0,  \\ (x,t) \in (0,+\infty) \times [0,T)
\end{multline}
with the boundary condition 
\[u(x,t)=\beta(x,t), \quad (x,t) \in \partial_{p} \left((0,+\infty) \times [0,T) \right),\] where $A \subset \mathbb{R}^{m}$ is a compact set and 
\begin{gather*}
b,h,l:\mathbb{R}\times [0,T]\times A \to \mathbb{R},\quad \sigma: \mathbb{R}\times [0,T] \to \mathbb{R}, \\ \beta: \partial_{p} \left((0,+\infty) \times [0,T) \right)\to \mathbb{R}
\end{gather*}
are continuous functions with further regularity assumptions to be specified later on.
The symbol $\partial_{p}$ denotes the parabolic boundary, i.e.,
\[
 \partial_{p} \left((0,+\infty) \times [0,T) \right)= ([0,+\infty) \times \{T\}) \cup (\{0\} \times [0,T]).
\] 

We emphasize that equation \eqref{first:eq} is only semilinear and we do not assume $\alpha$--dependence for the coefficient $\sigma$. We believe some results are available for the general equation, but the methodology we would like to present works only for semilinear equations. Besides, our results are sufficient to address the problems arising from our practical motivations.     

It should be noticed that HJB equation \eqref{first:eq} is naturally linked with a stochastic control problem. Namely, we can  assume that the controller has at his disposal a family $\{\mathcal{A}_{t}\}_{t \geq 0}$, where $\mathcal{A}_{t}$ is an admissible control set, that is, the set of all progressively measurable processes $\alpha=\{\alpha_{s}\}_{s \geq t}$ taking values in the compact set $A$. Finally, let the functions $b$ and $\sigma$ be such that the initial value problem 
\[
\begin{cases}
dX_{s}^{\alpha}= b(X_{s}^{\alpha},s,\alpha_{s})\;ds + \sigma(X_{s}^{\alpha},s) \; dW_{s}{,} \\ 
X_{s}^{\alpha} = x,
\end{cases}
\]
admits a unique strong solution, which we denote by $\{X_{s}^{\alpha}(x,t)\}_{s \geq t}$. The controller's aim is to maximize the reward function
\begin{multline} \label{control_problem}
\mathcal{J}^{\alpha} (x,t) := \ee \biggl[ \int_{t}^{T \wedge \tau^{\alpha}(x,t)} e^{\int_{t}^{s} h(X_{k}^{\alpha}(x,t),k,\alpha_{k})\; dk } l(X_{s}^{\alpha}(x,t),s,\alpha_{s})\; ds \\ + e^{\int_{t}^{T \wedge \tau^{\alpha}(x,t)} h(X_{k}^{\alpha}(x,t),k,\alpha_{k})\; dk }\beta(X^{\alpha}_{T \wedge \tau^{\alpha}(x,t)}(x,t) ,T \wedge \tau^{\alpha}(x,t))\biggr]
\end{multline}
with respect to an admissible control set $\mathcal{A}_{t}$, where
\[
\tau^{\alpha}(x,t):= \inf\{s \geq t |\; X_{s}^{\alpha} (x,t)\leq 0\}, \quad x \geq 0.
\]
The function $l$ represents the running reward, $\beta$ -- the terminal reward, $h$ -- the discount rate and the stopping time $\tau^{\alpha}$ represents the moment of termination.

Let the function $V$ denote the controller value function, that is, $V(x,t):=\sup_{\alpha \in \mathcal{A}_{t}} \mathcal{J}^{\alpha}(x,t)$. Once we prove that \eqref{first:eq} admits a bounded classical solution $u$, it is only a matter of applying the standard verification results (e.g. Fleming and Soner \cite[Theorem 3.1, Chapter IV]{FS}) to find a solution to the above stochastic control problem and prove that $u(x,t)=V(x,t)$.

In our paper, we treat equation \eqref{first:eq} even in greater generality. Namely, it should be noticed that equation \eqref{first:eq} can be written in the form 
\[
\begin{aligned} 
&D_{t} u+\tfrac{\sigma^{2}(x,t)}{2}D_{x}^{2}u+ H_{max}(D_{x}u,u,x,t)=0, \; (x,t) \in (0,+\infty) \times [0,T),\\
 &u(x,t)=\beta(x,t),  \qquad (x,t) \in \partial_{p} \left( (0,+\infty) \times [0,T) \right),    
\end{aligned}\]
where
\[H_{max}(p,u,x,t):= \max_{\alpha \in A}  \left(b(x,t,\alpha)p +  h(x,t,\alpha) u+ l(x,t,\alpha)\right).\]
 Therefore, considering other potential applications of our results, we have decided to postpone  specifying regularity assumptions on the functions $b$, $h$, and $l$ to the end of  Section 4, and to first address the general problem of finding a classical solution to
 \begin{equation} \label{second}
\begin{aligned}
&D_{t} u + \tfrac{\sigma^{2}(x,t)}{2}  D_{x}^{2} u + H(D_{x}u, u, x, t) = 0, \; (x,t) \in (0,+\infty) \times [0,T), \\
&u(x,t) = \beta(x,t),  \qquad (x,t) \in \partial_{p} \left( (0,+\infty) \times [0,T) \right),
\end{aligned}
\end{equation} 
where
\begin{gather*}
\sigma:\mathbb{R} \times [0,T] \to \mathbb{R}, \quad H:\mathbb{R} \times \mathbb{R} \times \mathbb{R} \times [0,T] \to \mathbb{R}, \\ \beta: \partial_{p} \left((0,+\infty) \times [0,T) \right)\to \mathbb{R}
\end{gather*}
are general functions with the following regularity assumptions.

\begin{itemize}
\item[{\bf A1)}]
The function $\sigma$ is bounded, bounded away from zero and Lipschitz continuous, i.e., there exists a constant $L_{\sigma}>0$ such that for all $x,\bar{x} \in \mathbb{R}$, $t,\bar{t} \in [0,T]$
\[
|\sigma(x,t)-\sigma(\bar{x},\bar{t})|\leq L_{\sigma} \left(|x-\bar{x}|+ |t-\bar{t}| \right).
\]

 \item[{\bf A2)}] The function $\beta$ is bounded and Lipschitz continuous, i.e., there exists a constant $L_{\beta}>0$ such that for all $(x,t), (\bar{x},\bar{t}) \in \partial_{p} \left((0,+\infty) \times [0,T) \right)$ 
 \[
 |\beta(x,t)-\beta(\bar{x},\bar{t})|\leq L_{\beta}\left(|x-\bar{x}|+ |t-\bar{t}| \right).
 \]
 
 \item[{\bf A3)}] The function $H$ is Lipschitz continuous on compact subsets, i.e., for each compact set $U \subset \mathbb{R}^{3} \times [0,T]$ there is a constant $L_{U}$ such that
for all  $(p,u,x,t)$, $(\bar{p},\bar{u},\bar{x},\bar{t}) \in U$  we have
\[
|H(p,u,x,t)-H(\bar{p},\bar{u},\bar{x},\bar{t})| \leq L_{U} \left(|p-\bar{p}| +|u-\bar{u}|+|x-\bar{x}|+|t-\bar{t}|\right).
\]
 Moreover, there exists a constant $K>0$ such that for all $(p,u,x,t)$, $(\bar{p},\bar{u},x,t) \in \mathbb{R}^{3}\times [0,T]$ we have
\begin{equation}\label{warH1}
\begin{aligned}
&|H(p,u,x,t)| \leq K \left(1+|u|+|p|\right),\\
&|H(p,u,x,t)-H(\bar{p},\bar{u},x,t)| \leq K \left(|u-\bar{u}| + |p - \bar{p}|\right). 
\end{aligned}
\end{equation}

 \item[{\bf A4)}]There exists a constant $L>0$ such that for all $(x,t),(\bar{x},t) \in [0,+\infty) \times [0,T]$
 \begin{equation} \label{property:L}
\left| \ee (\tau(x,t) \wedge T) - \ee (\tau(\bar{x},t) \wedge T) \right | \leq L|x-\bar{x}|,
\end{equation}
where
\[ \tau (x, t): = \inf \{s \geq t | \; X_{s}(x,t)\leq 0\}
\]
and 
\[
dX_{s}=\sigma(X_{s},s)\; dW_{s}, \quad X_{t}=x.
\]
\end{itemize}

Condition \textbf{A1)} guarantees the well-posedness and essential properties of the solutions to the stochastic differential equations and linear parabolic equations studied in this paper. Assumption \textbf{A2)} provides the regularity of the boundary data, which is important (together with \textbf{A1)} and \textbf{A4)}) for obtaining a global gradient bound for the solution to our PDE. Assumption \textbf{A3)} is fundamental for establishing the well-posedness and regularity of solutions to our nonlinear PDE.

In Section~5, we present a proof of the implication \textbf{A1)} $\Rightarrow$ \textbf{A4)} in the homogeneous case, under the additional assumption that $\sigma \in C^{1}(\mathbb{R})$ is bounded, bounded away from zero, and its derivative $D_{x}\sigma$ is bounded and locally Lipschitz continuous. Under assumption \textbf{A1)}, the existence of a bounded classical solution to the problem
\[
\begin{cases}
D_t u + \tfrac{\sigma^2(x,t)}{2}  D_x^2 u + 1 = 0, & \quad (x,t) \in (0,+\infty) \times [0,T), \\
u(x,t) = 0, & \quad (x,t) \in \partial_p \left( (0,+\infty) \times [0,T) \right),
\end{cases}
\]
implies that condition \textbf{A4)} is equivalent to the boundedness of the derivative \( D_x u \) of this solution. This follows directly from the Mean-Value Theorem and the Feynman--Kac representation (see Proposition~\ref{Feynman}):
\[
u(x,t) = \mathbb{E}[\tau(x,t) \wedge T] - t,
\]
where \[ \tau(x,t) := \inf\{s \geq t|\; X_s(x,t) \leq 0\}.\]

As was mentioned in the Introduction, an additional advantage of our work is that it provides a procedure for generating an approximating sequence that converges to the solution. In subsequent pages, we employ a fixed-point argument to guarantee relatively fast uniform convergence of the sequence on the entire space.

\section{Main result}

We start this section by formulating our main theorem.
\begin{thm} \label{main} Assume that conditions {A1)}--{A4)} are satisfied. Then there exists a classical solution  $u \in \mathcal{C}^{2,1}((0,+\infty) \times [0,T)) \cap \mathcal{C}([0,+\infty) \times [0,T])$ to 
\eqref{second}, which in addition is bounded together with $D_{x} u$.
\end{thm}
In Remark \ref{unique} we also address the uniqueness of the solution to \eqref{second}. The proof of the above-mentioned theorem is given at the end of this section. 
It is conducted by applying the fixed point method which involves constructing a recursive sequence consisting of solutions to an appropriate linear equation of the form
\begin{equation} \label{firsteq}
 \begin{aligned}
  &D_{t}u+ \tfrac{\sigma^{2}(x,t)}{2}D^{2}_{x}u + f(x,t)=0,  \; &(x,t) \in (0,+\infty) \times [0,T), \\
&u(x,t)=\beta(x,t), \; &(x,t) \in  \partial_{p} \left((0,+\infty) \times [0,T)\right).
\end{aligned}
\end{equation}
The sequence is further shown to converge to the solution of equation \eqref{second}.
So, first, we present some preparatory results on the solution to \eqref{firsteq}.

Let $\mathcal{C}_{b}^{1,0}$ \label{page} be the space of all bounded and continuous functions on the set $[0,+\infty) \times [0,T)$ for which the first derivative wrt. the variable $x$ exists on the set $(0,+\infty) \times [0,T)$ and is continuous and bounded. The space is endowed with the family of norms parametrised by $\kappa>0$
\[
\|u\|_{\kappa}: = \sup_{(x,t) \in [0,+\infty) \times [0,T)} e^{-\kappa(T-t)} |u(x,t)| +  \sup_{(x,t) \in (0,+\infty) \times [0,T)} e^{-\kappa(T-t)} |D _{x} u(x,t)|.
\]
Note that the space $\mathcal{C}_{b}^{1,0}$ together with $\|\cdot\|_{\kappa}$ forms a Banach space. The time-weighted norms are well known in the theory of differential equations as excellent tools to construct contractions on the entire interval $[0,T]$ in the existence proof. Our inspiration to use such norm comes from Becherer and Schweizer \cite{Becherer}, who employ time--weighted norms, but without the gradient component.

We also introduce the subspace $\mathcal{C}^{1+,0+}_{b,loc} \subset \mathcal{C}_{b}^{1,0}$ consisting of all functions $u \in \mathcal{C}_{b}^{1,0}$ for which the functions $u$ and $D_{x} u(x,t)$ are H\"older continuous on compact subsets of $(0,+\infty)\times [0,T)$ and globally bounded. There is no need to introduce a separate norm for $\mathcal{C}^{1+,0+}_{b,loc}$.
Furthermore, we use the space $\mathcal{C}_{b}^{0}$ consisting of all functions that are continuous on the set $(0,+\infty)\times [0,T)$ and bounded. This space is considered together with the norms
\begin{align*}
\|u\|^{0}:&= \sup_{(x,t) \in [0,+\infty)  \times [0,T)}  |u(x,t)|, \\ \|u\|_{\kappa}^{0}:&= \sup_{(x,t) \in [0,+\infty)  \times [0,T)} e^{-\kappa(T-t)} |u(x,t)|, \quad \kappa>0.
\end{align*}
If the function $u$ is continuous and bounded on the set $\mathbb{R} \times [0,T]$ we use the following notation
\begin{align*}
\|u\|^{0,\mathbb{R}}:&= \sup_{(x,t) \in \mathbb{R}  \times [0,T]}  |u(x,t)|, \\ \|u\|_{\kappa}^{0,\mathbb{R}}:&= \sup_{(x,t) \in \mathbb{R}  \times [0,T]} e^{-\kappa(T-t)} |u(x,t)|,\quad \kappa>0.
\end{align*}

In the proposition below we show the existence of the solution to \eqref{firsteq}. The proof includes the original idea for establishing  continuity up to the boundary for the limit of a sequence of functions.
\begin{prop} \label{approx}
Let conditions A1) and A2) be satisfied and let the function $f:(0,+\infty) \times [0,T) \to \mathbb{R}$ be H\"{o}lder continuous on compact subsets and bounded. Then there exists a bounded classical solution $u \in \mathcal{C}^{2,1}((0,+\infty) \times [0,T)) \cap \mathcal{C}([0,+\infty) \times [0,T])$ to \eqref{firsteq}, such that for every compact set $G \times I\subset (0,+\infty) \times [0,T)$ there exists $l \in (0,1]$ such that
$\| u\|_{\mathcal{C}^{2+l,1+l/2}(G \times I)} < +\infty$.
\end{prop}

\begin{proof}
We would like to use already existing results on linear equations provided by Rubio \cite{Rubio}. In fact we do not fully exploit the scope of Rubio's result, as we consider only equations with bounded coefficients, which is only special case of  Rubio's more general framework that accommodates unbounded coefficients. 

To use the result, we should first approximate $f$ using a sequence of Lipschitz continuous functions. The definition of the sequence is adapted to our setting from Evans \cite[Appendix C.5]{Evans}. The remainder of the proof is our own original contribution.

Let us consider first a test function 
\begin{equation} \label{test_function}
\eta(x,t):=\begin{cases}
Ce^{\frac{1}{|(x,t)|^{2}-1}}, & \text{if} \; |(x,t)| <1,  \\
0, &\text{otherwise},
\end{cases}
\end{equation}
where the constant  $C$ is such that 
\begin{equation} \label{integrate_one}
\int_{\mathbb{R}^{2}}\eta(z) \; dz=1, \quad \text{($z\in \mathbb{R}^{2}$ is a shorthand notation for $(x,t)$)}.
\end{equation}
We may define a family 
\[
\eta_{\varepsilon}(x,t):= \frac{1}{\varepsilon^2}{\eta}\left(\frac{x}{\varepsilon},\frac{t}{\varepsilon}\right).
\]
The definition of the function $f$ can be extended on the entire $\mathbb{R} \times \mathbb{R}$ in the following way
$f(x,t):=f(x,0), \; t \in [-1,0)$ and $f=0$ outside of the set $(0,+\infty) \times [-1,T)$. Now, we choose an increasing sequence $(R_n,\; n \in \mathbb{N})$, $R_n \to +\infty$ as $n \to +\infty$ and define
\[
\psi_{n} (x,t):=\begin{cases}
f(x,t), & \text{if} \; (x,t) \in (0,R_{n}) \times [-1,T),\\
0, &\text{otherwise}
\end{cases}
\]
and
\begin{align*}
f_{n}(z):&= \int_{\mathbb{R}^{2}} \eta_{\varepsilon_{n}}(z-\zeta) \psi_{n}(\zeta) \; d \zeta = \int_{\mathbb{R}^{2}} \eta_{\varepsilon_{n}}(\zeta) \psi_{n}(z-\zeta) \; d \zeta \\ &=\int_{(0,R_{n}) \times [-1,T)} \eta_{\varepsilon_{n}}(z-\zeta) \psi_{n}(\zeta) \; d \zeta=\int_{B(0,\varepsilon_{n})} \eta_{\varepsilon_{n}}(\zeta) \psi_{n}(z-\zeta) \; d \zeta,
\end{align*}
where $\varepsilon_{n}:=\frac{1}{n}$ and  $B(0,\varepsilon_{n})$ denotes the support of the function  $\eta_{\varepsilon_{n}}$ (a closed ball with the center in $0$ and the radius equals $\varepsilon_{n}$).
Notice that
\begin{enumerate}
\item[i)] $f_{n}$ is uniformly bounded (because $f$ is bounded and condition \eqref{integrate_one} holds),
\item[ii)] $f_{n}$ is globally Lipschitz continuous in $(x,t)$ on the set $[0,+\infty)\times[0,T]$ (because $\eta_{\varepsilon_{n}}$ is globally Lipschitz continuous and $\psi_{n}$ is integrable),
\item[iii)]$f_{n} \to f$ (locally uniformly  on the set $(0, +\infty) \times [0,T)$) (by the continuity of $f$ on the set $(0, +\infty) \times [0,T)$ and standard arguments, Evans \cite[Appendix C.5]{Evans}), 
\item[iv)]$f_{n}$ is H\"older continuous on compact subsets of $(0,+\infty)\times (-1,T)$ uniformly with respect to $n$. 
\end{enumerate}

To provide more evidence for the correctness of iv), we fix a compact set $U \subset (0,+\infty)\times (-1,T)$ and define
\[
(U - B(0,\varepsilon)):=\{z-\zeta|\;z \in U, \zeta \in B(0,\varepsilon)\}.
\]
  Note that there exists $n_{0} \in \mathbb{N}$ such that for all $n \geq n_{0}$ we have 
 \[(U - B(0,\varepsilon_{n})) \subset (U - B(0,\varepsilon_{n_{0}})) \subset (0,+\infty)\times [-1,T).\]
The function $f$ is H\"older continuous on $U - B(0,\varepsilon_{n_{0}})$, which means that there exist $K_{0}>0$, $l \in (0,1]$ such that for all $\xi,\bar{\xi} \in (U - B(0,\varepsilon_{n_{0}}))$
\begin{equation} \label{Holder}
|f(\xi)-f(\bar{\xi})| \leq K_{0}\|\xi-\bar{\xi}\|^{l}.
\end{equation}
Now, we fix $z,\bar{z} \in U$ and $n \geq n_{0}$.  
Using \eqref{Holder}, we arrive at
\begin{align*}
|f_n(z)-f_n(\bar{z})| &\leq \int_{B(0,\varepsilon_{n})} \eta_{\varepsilon_{n}}(\zeta) |\psi_{n}(z-\zeta)- \psi_{n}(\bar{z}-\zeta)|  \; d \zeta \\ &\leq K_{0}\int_{B(0,\varepsilon_{n})} \eta_{\varepsilon_{n}}(\zeta) \|z-\bar{z}\|^{l} \; d \zeta \leq K_{0} \|z-\bar{z}\|^{l}.
\end{align*} 

By Rubio \cite[Theorem 3.1]{Rubio} there exists a classical solution \[u_{n} \in   \mathcal{C}^{2+1,1+1/2}_{loc}((0,+\infty) \times [-1,T)) \cap \mathcal{C}([0,+\infty) \times [0,T]))\] to 
\begin{equation*} 
 \begin{cases}
  D_{t}u+ \tfrac{\sigma^{2}(x,t)}{2}D^{2}_{x}u + f_{n}(x,t)=0,  &\quad (x,t) \in (0,+\infty) \times [-1,T), \\
u(x,t)=\beta(x,t), &\quad (x,t) \in  \partial_{p} \left((0,+\infty) \times [-1,T)\right).
\end{cases}
\end{equation*}
In addition, we have the Feynman-Kac formula (see Rubio \cite[Theorem 3.1]{Rubio}), i.e.,
\begin{equation} \label{F-Kf}
u_{n}(x,t)= \ee \biggl[\beta(X_{T \wedge \tau(x,t)}(x,t),T \wedge \tau(x,t))  + \int_{t}^{T \wedge \tau(x,t)} f_{n} (X_{s}(x,t),s)\; ds \biggr].
\end{equation}

Now, according to Ladyzhenskaja et al. \cite[Chapter IV, Theorem 10.1]{Lady} for any two compact sets of the form $[\frac{1}{k},k]\times [0,T-\frac{1}{k}] \subset [\frac{1}{k+1},k+1] \times [-\frac{1}{k},T-\frac{1}{k}] \subset (0,+\infty)\times (-1,T)$, $k \in \mathbb{N},\; k \geq 2$ there exist  $M_{k}>0$ and $l \in (0,1]$ such that
\begin{multline*}
\|u_{n}\|_{\mathcal{C}^{2+l,1+l/2}\left([\frac{1}{k},k]\times [0,T-\frac{1}{k}]\right)} \\ \leq M_{k} \left(\|u_{n}\|_{C \left([\frac{1}{k+1},k+1] \times [-\frac{1}{k},T-\frac{1}{k}] \right)} + \|f_{n}\|_{\mathcal{C}^{l,l/2}\left([\frac{1}{k+1},k+1] \times [-\frac{1}{k},T-\frac{1}{k}]\right)} \right) \\  k,n \in \mathbb{N},\;k\geq 2.
\end{multline*}
Note that both $\|u_{n}\|_{C \left([\frac{1}{k+1},k+1] \times [-\frac{1}{k},T-\frac{1}{k}] \right)}$  and $\|f_{n}\|_{\mathcal{C}^{l,l/2}\left([\frac{1}{k+1},k+1] \times [-\frac{1}{k},T-\frac{1}{k}]\right)}$ are bounded (by (i), \eqref{F-Kf}, and (iv)). This leads to boundedness and equicontinuity of families $(u_{n},n \in \mathbb{N})$, $(D_x u_{n},n \in \mathbb{N})$, $(D_x^{2} u_{n},n \in \mathbb{N})$,
$(D_t u_{n},n \in \mathbb{N})$ on each set $[\frac{1}{k},k]\times [0,T-\frac{1}{k}]$. By the Arzela-Ascoli lemma (applied to the each set $[\frac{1}{k},k]\times [0,T-\frac{1}{k}]$ separately) and the standard diagonal argument, there exists a subsequence $(n_k, k \in \mathbb{N})$ such that $(u_{n_k},k \in \mathbb{N})$ converges to $u$ together with its first and second spatial derivatives and the first time derivative, uniformly in the H\"older norm, on each compact subset of $(0,+\infty) \times [0,T)$  and consequently $u$ is a solution to \eqref{firsteq} and $\|u\|_{\mathcal{C}^{2+l,1+l/2}(G\times I)}<+\infty$  on every compact set  $G\times I \subset (0,+\infty) \times [0,T)$. The boundedness of $u$ follows from \eqref{F-Kf}.

The next step is to prove the Feynman-Kac formula for the function $u$. We start by noting that 
\[
u_{n_k}(x,t)= \ee \left[\beta(X_{T \wedge \tau(x,t)}(x,t),T \wedge \tau(x,t))+ \int_{t}^{T \wedge \tau(x,t)} f_{n_k} (X_{s}(x,t),s)\; ds \right].
\]
The family $(f_{n}\; n \in \mathbb{N})$ is uniformly bounded, therefore, using the dominated convergence theorem and passing to the limit (as $k \to +\infty$), we get
\[
u(x,t) = \ee \left[\beta(X_{T \wedge \tau(x,t)}(x,t),T \wedge \tau(x,t))+ \int_{t}^{T \wedge \tau(x,t)} f(X_{s}(x,t),s)\; ds \right].
\]

Finally, we need to show the continuity up to the boundary of the function $u$. Note that the function $\beta$ is Lipschitz continuous, while $f$ is bounded, which means that for all $(\bar{x},\bar{t}) \in \partial_{p} \left((0,+\infty) \times [0,T)\right)$ we have
\begin{align*}
\left|u(x,t)- \beta(\bar{x},\bar{t})\right| \leq L_{\beta} \left( \ee\left|X_{T \wedge \tau(x,t)}(x,t) -
\bar{x}\right| +  \ee\left|T \wedge \tau(x,t) -\bar{t}\right| \right) \\ + \| f\|^{0}\ee\left|T \wedge \tau(x,t) -t\right|.
\end{align*}
Additionally, we have
\[
\ee\left|(T \wedge \tau(x,t)) -\bar{t}\right| \leq \ee\left|(T \wedge \tau(x,t)) -t\right|+ |t-\bar{t}|.
\]
Using the H\"older inequality and the It\^o isometry, we get 
\begin{multline*}
\ee \left|X_{T \wedge \tau(x,t)}(x,t) -\bar{x}\right| \leq |x-\bar{x}| + \left[\ee \int_{t}^{T \wedge \tau(x,t)} [ \sigma(X_{s}(x,t),s)]^2\; ds   \right]^{\frac{1}{2}} \\ \leq |x-\bar{x}|+ \|\sigma\|^{0}  \left[\ee\left( T \wedge \tau(x,t))-t \right) \right]^{\frac{1}{2}}.
\end{multline*}
Consequently, if $\lim_{(x,t) \to (\bar{x},\bar{t})}\ee \left(( T \wedge \tau(x,t))-t \right) = 0$, then the function $u$ is continuous up to the boundary. However, note that Rubio \cite[Theorem 3.1]{Rubio} proved the existence of a continuous solution to the Cauchy--Dirichlet problem
\[
 \begin{cases}
  D_{t} u+ \frac{1}{2}\sigma^2(x,t)D^{2}_{x}u +1=0,  &\quad (x,t) \in (0,+\infty) \times [0,T), \\
u(x,t)=0, &\quad (x,t) \in  \partial_{p} \left((0,+\infty) \times [0,T)\right),
\end{cases}
\]
which additionally has the Feynman--Kac representation: 
\[u(x,t)= \mathbb{E} \left[ \int_{t}^{T \wedge \tau(x,t)}1 \; ds\right]= \mathbb{E}\left(T \wedge \tau(x,t)\right) -t.\] 
The continuity of the function $u$ confirms that \[\lim_{(x,t) \to (\bar{x},\bar{t})}\ee \left(( T \wedge \tau(x,t))-t \right) = 0.\]
\end{proof}

We now present a version of the Feynman-Kac formula that will be used to prove subsequent results.

\begin{prop} \label{Feynman} Suppose that condition A1) is satisfied and 
\[u \in \mathcal{C}^{2,1}\left((0,+\infty) \times [0,T)\right) \cap \mathcal{C}\left([0,+\infty) \times [0,T]\right)\] is a bounded solution to
\begin{equation*} 
 \begin{cases}
  D_{t}u+ \tfrac{\sigma^{2}(x,t)}{2}D^{2}_{x}u + f(x,t)=0,  &\quad (x,t) \in (0,+\infty) \times [0,T), \\
u(x,t)=\beta(x,t), &\quad (x,t) \in  \partial_{p} \left((0,+\infty) \times [0,T)\right)
\end{cases}
\end{equation*}
where $f:(0,+\infty) \times [0,T) \to \mathbb{R}$ and $\beta:\partial_{p} \left((0,+\infty) \times [0,T)\right)\to \mathbb{R}$ are bounded and continuous functions. Then
\[
u(x,t) = \ee \left[\beta(X_{T \wedge \tau(x,t)}(x,t),T \wedge \tau(x,t))+ \int_{t}^{T \wedge \tau(x,t)} f(X_{s}(x,t),s)\; ds \right],
\]
where
\[ \tau (x, t): = \inf \{s \geq t | \; X_{s}(x,t)\leq 0\}
\]
and 
\[
dX_{s}=\sigma(X_{s},s)\; dW_{s}, \quad X_{t}=x.
\]
\end{prop}

\begin{proof}
Define
\[
\tau_{n}(x,t):=\inf \left \{s \geq t | \; X_{s}(x,t)\leq \frac{1}{n} \right \}, \quad n \in \mathbb{N}.
\]
By the standard Feynman--Kac formula (see the proof of Felming and Soner \cite[Chapter IV, Lemma 3.1, (ii)]{FS}), we have
\begin{multline*}
u(x,t) = \ee \biggl[u\left(X_{\left(T-\frac{1}{n}\right) \wedge \tau_{n}(x,t)}(x,t),\left(T-\frac{1}{n}\right) \wedge \tau_{n}(x,t)\right) \\+ \int_{t}^{\left(T-\frac{1}{n}\right) \wedge \tau_{n}(x,t)} f(X_{s}(x,t),s)\; ds \biggr].
\end{multline*}
Observe that $(\tau_{n}, n \in \mathbb{N})$ is an increasing sequence, bounded by $\tau(x,t)$. Therefore, there exists $\hat{\tau}(x,t):=\lim_{n \to +\infty} \tau_{n}(x,t)$ and $\hat{\tau}(x,t) \leq \tau(x,t)$. Additionally, we have
\[
0= \lim_{n \to + \infty} \frac{1}{n}\mathbf{1}_{\{ \hat{\tau}(x,t) < + \infty\}} = \lim_{n \to + \infty}\mathbf{1}_{\{ \hat{\tau}(x,t) < + \infty\}} X_{\tau_{n}(x,t)} =\mathbf{1}_{\{ \hat{\tau}(x,t) < + \infty\}}X_{\hat{\tau}(x,t)}.
\]
This yields $\hat{\tau}(x,t) = \tau(x,t)$.
Since $\lim_{n \to \infty} \tau_{n}(x,t) =\tau(x,t)$, we can apply the dominated convergence theorem  and pass to the limit under the expectation sign and obtain
  \[
u(x,t) = \ee \left[\beta(X_{T \wedge \tau(x,t)}(x,t),T \wedge \tau(x,t))+ \int_{t}^{T \wedge \tau(x,t)} f(X_{s}(x,t),s)\; ds \right].
\]

\end{proof}

The next step in our reasoning is to prove estimates for $\|u\|_{\kappa}$ where $u$ is a solution to \eqref{firsteq}.
For the proof we need first to find a proper estimate for the Cauchy problem
\begin{equation}\label{Cauchy0}
 \begin{cases}
D_{t} v+\tfrac{\sigma^{2}(x,t)}{2}D_{x}^{2}v+ f(x,t)=0, &\quad (x,t) \in \mathbb{R} \times [0,T), \\
v(x,T)=0, &\quad x \in \mathbb{R}. \\
\end{cases}
\end{equation} 
By $v_{f}$ we denote the unique bounded classical (class $C^{2,1}$) solution to \eqref{Cauchy0}. It exists due to Heath and Schweizer  \cite[Theorem 1 and a comment on condition A3') (pages 950-951)]{Heath} for $\sigma$ being Lipschitz continuous, bounded and bounded away from zero and for $f$ being H\"older continuous on compact subsets of $\mathbb{R} \times [0,T]$.
We have the following proposition.

\begin{prop}\label{norm:estimates:Cauchy}
Under A1) there exists a constant $M$ such that for all $\kappa >0$ and all $f$ being H\"{o}lder continuous on compact subsets of $\mathbb{R} \times [0,T]$ and bounded we have
\[\|D_{x} v_{f}\|_{\kappa}^{0,\mathbb{R}} \leq \dfrac{M}{\sqrt[3]{\kappa}}\|f\|_{\kappa}^{0,\mathbb{R}}.\]
\end{prop}

\begin{proof}
In the proof, we use the theory of fundamental solutions for parabolic equations. The inspiration to use it in estimation of $D_{x}v_{f}$ comes from Bychowska and Leszczy\'nski \cite{Bychowska}, but the proof itself is our original contribution. It presents an original idea how to avoid differentiation under the integral when the function has a singularity. The fundamental solution is denoted by
$\Gamma(x,t,z,s)$.
Recall that
there exist $c,C>0$ such that 
\begin{align} 
|\Gamma(x,t,y,s)| &\leq \frac{C}{(s-t)^{1/2}} \exp \biggl(-c \frac{|y-x|^{2}}{(s-t)}\biggr), \notag \\ |D_{x} \Gamma(x,t,y,s)| &\leq \frac{C}{(s-t)} \exp \biggl(-c \frac{|y-x|^{2}}{(s-t)}\biggr), \notag \quad \quad s>t, \; x,y \in \mathbb{R}
\end{align}
(see Friedman \cite[ Chapter 1, eq. (6.12), eq. (6.13)]{Friedman}).

Instead of estimating $D_{x} v_{f}$ itself, we find a bound by estimating the Lipschitz constant in $x$ for $v_{f}$. First, we fix $x,\bar{x} \in \mathbb{R}$, $x<\bar{x}$, and $t \in [0,T)$. Using elementary facts from the linear parabolic equation theory, we get
\begin{align*}
|v_{f}(x,t)-v_{f}(\bar{x},t)|  &= \left| \mathbb{E} \int_{t}^{T}  \left(f(X_{s}(x,t),s)-f(X_{s}(\bx,t),s) \right) \; ds \right| \\
&= \left| \int_{t}^{T} \mathbb{E} \left(f(X_{s}(x,t),s)-f(X_{s}(\bx,t),s) \right) \; ds \right|
\\ &= \left| \int_{t}^{T} \int_{\mathbb{R}} f(z,s) \left(\Gamma(x,t,z,s) - \Gamma(\bar{x},t,z,s)\right)\; dz \; ds \right| \\ &\leq \int_{t}^{T} \int_{\mathbb{R}} |f(z,s)| \left|\Gamma(x,t,z,s) - \Gamma(\bar{x},t,z,s)\right|\; dz \;ds,
\end{align*}
where
\[
dX_{s}=\sigma(X_{s},s)\;dW_{s}, \quad X_{t}=x.
\]
Note that
\begin{align*}
&\int_{\mathbb{R}} |f(z,s)| \left|\Gamma(x,t,z,s) - \Gamma(\bar{x},t,z,s)\right|\; dz \\ & =  \int_{\mathbb{R} \setminus (x,\bar{x})} |f(z,s)| \left|\Gamma(x,t,z,s) - \Gamma(\bar{x},t,z,s)\right|\; dz \\ &\quad + \int_{x}^{\bar{x}} |f(z,s)| \left|\Gamma(x,t,z,s) - \Gamma(\bar{x},t,z,s)\right|\; dz\\  & \leq \int_{\mathbb{R} \setminus (x,\bar{x})} |f(z,s)| \left|\Gamma(x,t,z,s) - \Gamma(\bar{x},t,z,s)\right|\; dz \\ &\quad + \int_{x}^{\bar{x}} |f(z,s)| \left|\Gamma(x,t,z,s) + \Gamma(\bar{x},t,z,s)\right|\; dz. 
\end{align*}
According to the mean value theorem, there exists $x^{*}(t,z,s) \in (x,\bar{x})$ such that 
\[
\Gamma(x,t,z,s) - \Gamma(\bar{x},t,z,s)=D_{x}\Gamma(x^{*}(t,z,s),t,z,s)(x-\bar{x}). 
\]
Furthermore, we already know that there exist $c,C>0$ such that
\[
|D_{x}\Gamma(x^{*}(t,z,s),t,z,s)| \leq \frac{C}{(s-t)} \exp \biggl(-c \frac{|z-x^{*}(t,z,s)|^{2}}{(s-t)}\biggr).
\]
 We have $x^{*}(t,s,z) \in [x,\bar{x}]$ and thus
 \begin{multline*}
 \exp \biggl(-c \frac{|z-x^{*}(t,z,s)|^{2}}{(s-t)}\biggr)\leq \exp \biggl(-c \frac{|z-x|^{2}}{(s-t)}\biggr) + \exp \biggl(-c \frac{|z-\bar{x}|^{2}}{(s-t)}\biggr),  \\ z \in \mathbb{R} \setminus (x,\bar{x}).
 \end{multline*} We can also use the fact
\[ \int_{\mathbb{R}} \frac{1}{\sqrt{s-t}} \exp \biggl(-c \frac{|z-x|^{2}}{(s-t)}\biggr)\; dz=\frac{\sqrt{2 \pi }}{\sqrt{2c}}, \quad x \in \mathbb{R}.
\] 
 In summary, we have the following inequality
\begin{align*}
 |v_{f}(x,t)-v_{f}(\bar{x},t)|&\leq 2C\left(\frac{\sqrt{2 \pi }}{\sqrt{2c}}+1\right) \left[\int_{t}^{T} \frac{1}{\sqrt{s-t}} e^{\kappa(T-s)} ds \right]\| f \| _{\kappa}^{0,\mathbb{R}}|x-\bar{x}|.   
\end{align*}
Multiplying both the sides by $e^{-\kappa (T-t)}$ and using the H\"older inequality, we obtain
\begin{align*}
&e^{-\kappa (T-t)} |v_{f}(x,t)-v_{f}(\bar{x},t)| \\ &\leq 2C\left(\frac{\sqrt{2 \pi }}{\sqrt{2c}}+1\right) \left[\int_{t}^{T} \frac{1}{\sqrt{s-t}} e^{\kappa(t-s)} ds \right]\| f \| _{\kappa}^{0,\mathbb{R}}|x-\bar{x}| \\ 
\\ 
& \leq 2C\left(\frac{\sqrt{2 \pi }}{\sqrt{2c}}+1\right)\| f \| _{\kappa}^{0,\mathbb{R}} \biggl(\int_{t}^{T} (s-t)^{-\frac{3}{4}} ds\biggr)^{\frac{2}{3}} \biggr( \int_{t}^{T} e^{3\kappa(t-s)} ds \biggr)^{\frac{1}{3}}|x-\bar{x}|.
\end{align*}
Additionally, we have
\[
\left[\int_{t}^{T} e^{3\kappa(t-s)} ds\right]^{\frac{1}{3}} =\left[\frac{1}{3 \kappa} \left[1- e^{3\kappa(t-T)}\right]\right]^{\frac{1}{3}} \leq \frac{1}{\sqrt[3]{3 \kappa}},
\]
which completes the proof.
\end{proof}

Let $u_{f} \in  \mathcal{C}^{2,1}((0,+\infty) \times [0,T)) \cap \mathcal{C}([0,+\infty) \times [0,T])$ denote  the unique bounded classical solution to 
\begin{equation*} 
 \begin{cases}
  D_{t}u+ \tfrac{\sigma^{2}(x,t)}{2}D^{2}_{x}u + f(x,t)=0,  &\; (x,t) \in (0,+\infty) \times [0,T), \\
u(x,t)=0, &\; (x,t) \in  \partial_{p} \left((0,+\infty) \times [0,T)\right).
\end{cases}
\end{equation*}
It exists due to Proposition \ref{approx} for $\beta \equiv 0$ and for $f$ being H\"older continuous on compact subsets of $(0,+\infty) \times [0,T)$. The uniqueness follows from the Feynman--Kac representation (Proposition \ref{Feynman}). 

The proof of the next result presents a nice and original idea for transferring gradient bound from functions defined on the entire domain $\mathbb{R} \times [0,T)$ to those defined on $(0,+\infty) \times [0,T)$.

\begin{prop} \label{norm:estimates}
Assume that conditions A1) and A4) are satisfied. Then there exists a constant $M>0$ such that for all $\kappa>1$ and for all functions $f$ being H\"older continuous on compact subsets of $\mathbb{R} \times [0,T]$ and bounded we have
\[\| u_{f}\|_{\kappa} \leq \dfrac{M}{\sqrt[3]{\kappa}}\|f\|_{\kappa}^{0,\mathbb{R}}.\]
\end{prop}

\begin{proof} Due to the Feynman - Kac representation (Proposition \ref{Feynman}), we have
\[u_{f}(x,t)=\et \int_{t}^{T \wedge\tau(x,t)} f(X_{s}(x,t),s)\; ds. \]
Hence,
\begin{multline*}
e^{-\kappa(T-t)}u_{f}(x,t)= \et e^{-\kappa(T-t)}\int_{t}^{T \wedge\tau(x,t)}e^{\kappa(T-s)} e^{-\kappa(T-s)}f(X_{s}(x,t),s)\; ds \\ \leq \|f\|_{\kappa}^{0} e^{-\kappa(T-t)}\int_{t}^{T}e^{\kappa(T-s)} \; ds = \frac{1}{\kappa}\|f\|_{\kappa}^{0} e^{-\kappa(T-t)}(e^{\kappa(T-t)} -1) \leq \frac{1}{\kappa}\|f\|_{\kappa}^{0,\mathbb{R}}.
\end{multline*}

Once again, a bound for the derivative $D_{x} u_{f}$ is determined by estimating the Lipschitz constant. Fix $x,\bar{x} \in [0,+\infty]$ and assume $x \geq \bar{x}$. In particular, the last assumption implies
\[
P\left(\forall_{s \in [t,T]} \; X_{s}(x,t) \geq  X_{s}(\bx,t)\right)=1
\]
(by the comparison theorem -- see e.g. Ikeda and Watanabe \cite[Theorem 1.1]{Ikeda}), and
therefore $\tau(x,t) \geq \tau(\bx,t)$ a.s. 

We have
\begin{align*}
|u_{f}(x,t)-& u_{f}(\bar{x},t)| \\ \leq &\left| \ee  \int_{t}^{T \wedge\tau(x,t)} f(X_{s}(x,t),s)\;ds - \ee  \int_{t}^{T \wedge\tau(x,t)}f(X_{s}(\bar{x},t),s)\;ds \right|  \\ & \quad +  \left| \ee \int_{t}^{T \wedge\tau(x,t)} f(X_{s}(\bar{x},t),s)\; ds - \ee \int_{t}^{T \wedge\tau(\bx,t)} f(X_{s}(\bar{x},t),s) \; ds \right|\\ & \qquad =:I_{1} + I_{2}.
\end{align*}

 First, to estimate $I_1$, consider the Cauchy problem \eqref{Cauchy0}. Note that for the solution $v_{f}$ we have obtained (Proposition \ref{norm:estimates:Cauchy}) the inequality
 \begin{equation}\label{inequality:from:Cauchy}
 \|D_{x} v_{f}\|_{\kappa}^{0} \leq \dfrac{M}{\sqrt[3]{\kappa}}\|f\|_{\kappa}^{0,\mathbb{R}}.
 \end{equation}
Additionally, the It\^o formula yields  
 \begin{equation} \label{fromCauchy}
\ee v_{f}(X_{T \wedge \tau(x,t)}(x,t), T \wedge \tau(x,t))=v_{f}(x,t)\\ - \ee \int_{t}^{T \wedge \tau(x,t)} f(X_{s}(x,t),s)\; ds
 \end{equation}
 and
 \begin{equation} \label{fromCauchy2}
\ee v_{f}(X_{T \wedge \tau(x,t)}(\bar{x},t), T \wedge \tau(x,t))=v_{f}(\bar{x},t) \\- \ee \int_{t}^{T \wedge \tau(x,t)} f(X_{s}(\bar{x},t),s)\; ds.
 \end{equation}
Using \eqref{fromCauchy} and \eqref{fromCauchy2}, we obtain the following

\begin{align*}
I_{1}&=\left| \ee  \int_{t}^{T \wedge\tau(x,t)} f(X_{s}(x,t),s)\; ds - \ee  \int_{t}^{T \wedge\tau(x,t)}f(X_{s}(\bar{x},t),s)\; ds \right|\\
&=\left|  \ee  \int_{t}^{T \wedge\tau(x,t)}f(X_{s}(\bar{x},t),s)\; ds - \ee  \int_{t}^{T \wedge\tau(x,t)} f(X_{s}(x,t),s)\; ds \right|
  \\ &= \begin{aligned}[t]\biggl|\ee v_{f}(X_{T \wedge \tau(x,t)}(x,t),  T \wedge \tau(x,t))-\ee &v_{f}(X_{T \wedge \tau(x,t)}(\bar{x},t),  T \wedge \tau(x,t))\\ 
 &+ v_{f}(\bar{x},t)- v_{f}(x,t) \biggr| \end{aligned}\\
&\leq \begin{aligned}[t] \biggl|\ee v_{f}(X_{T \wedge \tau(x,t)}(x,t), T \wedge \tau(x,t))-\ee & v_{f}(X_{T \wedge \tau(x,t)}(\bar{x},t),  T \wedge \tau(x,t))\biggr|\\ &+ \left|v_{f}(\bar{x},t)- v_{f}(x,t) \right|. \end{aligned}
\end{align*}
Note that the function $v_{f}$ satisfies \eqref{inequality:from:Cauchy} and therefore, 
\[
\left|v_{f}(\bar{x},t)- v_{f}(x,t) \right| \leq e^{\kappa(T-t)}\dfrac{M}{\sqrt[3]{\kappa}}\|f\|_{\kappa}^{0,\mathbb{R}}|x-\bar{x}|.
\]
In addition,
\begin{multline*}
\left|\ee v_{f}(X_{T \wedge \tau(x,t)}(x,t), T \wedge \tau(x,t))-\ee v_{f}(X_{T \wedge \tau(x,t)}(\bar{x},t), T \wedge \tau(x,t))\right| \\ \leq \dfrac{M}{\sqrt[3]{\kappa}}\|f\|_{\kappa}^{0,\mathbb{R}} \; \ee e^{\kappa(T-T \wedge \tau(x,t))}|X_{T \wedge \tau(x,t)}(x,t)- X_{T \wedge \tau(x,t)}(\bar{x},t)|.
\end{multline*}
The condition $x \geq \bar{x}$ implies 
\[
X_{T \wedge \tau(x,t)}(x,t) \geq X_{T \wedge \tau(x,t)}(\bar{x},t)
\]
(e.g. Ikeda and Watanabe \cite[Theorem 1.1]{Ikeda}) and consequently 
 \begin{multline*}
 \ee |X_{T \wedge \tau(x,t)}(x,t)- X_{T \wedge \tau(x,t)}(\bar{x},t)| \\ = \ee X_{T \wedge \tau(x,t)}(x,t)- \ee X_{T \wedge \tau(x,t)}(\bar{x},t)= x- \bx= |x-\bx|.
 \end{multline*}
In summary, we get
 \[ 
 e^{-\kappa(T-t)}I_{1} \leq \dfrac{2M}{\sqrt[3]{\kappa}}\|f\|_{\kappa}^{0,\mathbb{R}}|x-\bx|.
 \]
 
For the second addend $I_{2}$ we have
\begin{align*} I_{2} &= \left| \ee \int_{\tau(\bx,t) \wedge T}^{\tau(x,t) \wedge T} f(X_{s}(\bar{x},t),s) \; ds\right| \\ &=\left| \ee \int_{\tau(\bx,t) \wedge T}^{\tau(x,t) \wedge T} e^{\kappa(T-s)} e^{-\kappa(T-s)}f(X_{s}(\bar{x},t),s) \; ds\right|  \\ &\leq  \|f\|_{\kappa}^{0,\mathbb{R}} \ee \left|\int^{\tau(x,t) \wedge T}_{\tau(\bar{x},t)\wedge T} e^{\kappa(T-s)} \;  ds \right| \\ &=  \frac{1}{\kappa}\|f\|_{\kappa}^{0,\mathbb{R}} \ee \left|e^{\kappa(T-T \wedge \tau(\bar{x},t))}- e^{\kappa(T-T \wedge \tau(x,t))}\right|
\end{align*}
and finally,
\begin{multline*}
e^{-\kappa (T-t)} I_{2}  \leq \frac{1}{\kappa}\|f\|_{\kappa}^{0,\mathbb{R}}\ee \left[e^{\kappa(t-T \wedge \tau(\bar{x},t))}- e^{\kappa(t-T \wedge \tau(x,t))}\right] \\ \leq \frac{1}{\kappa}\|f\|_{\kappa}^{0,\mathbb{R}} \ee \left|(T \wedge \tau(\bar{x},t)) -    (T \wedge \tau(x,t))\right| \leq  \frac{1}{\kappa}\|f\|_{\kappa}^{0,\mathbb{R}}|x-\bar{x}|,
\end{multline*}
which ends the proof.
\end{proof}

\begin{rem} \label{to_norm_estimates}
 The assertion of Proposition \ref{norm:estimates} is valid not only for functions $f:\mathbb{R} \times [0,T) \to \mathbb{R}$, which are H\"{o}lder continuous on compact subsets and bounded, but also for functions $f:(0,+\infty)\times [0,T) \to \mathbb{R}$ that are H\"{o}lder continuous on compact subsets and bounded.
 \end{rem} 
\begin{proof} 
 To show this, it is sufficient to approximate the function $f$ by the sequence $(f_{n},n \in \mathbb{N})$ defined in the proof of Proposition \ref{approx}.
By Proposition \ref{norm:estimates}, we have
\begin{equation} \label{forf}
 \| u_{f_{n}}\|_{\kappa} \leq \dfrac{M}{\sqrt[3]{\kappa}}\|f_{n}\|_{\kappa}^{0,\mathbb{R}}\leq \dfrac{M}{\sqrt[3]{\kappa}}\|f\|_{\kappa}^{0}.
\end{equation}
The second inequality follows from the definition of the function $f_{n}$. Namely, 
\begin{align*}
e^{-\kappa(T-t)}f_{n}(x,t)&= \int_{\mathbb{R}^{2}} \eta_{\varepsilon_{n}}(\zeta)e^{-\kappa(T-t)} \psi_{n}((x,t)-\zeta) \; d \zeta \leq 
\int_{\mathbb{R}^{2}} \eta_{\varepsilon_{n}}(\zeta)\|f\|_{\kappa}^{0} \; d \zeta \\ &= \|f\|_{\kappa}^{0}\int_{\mathbb{R}^{2}} \eta_{\varepsilon_{n}}(\zeta)\; d \zeta =\|f\|_{\kappa}^{0}.
\end{align*}
Inequality \eqref{forf} implies that for all $(x,t) \in (0,+\infty) \times [0,T)$ 
\[
e^{-\kappa(T-t)}|u_{f_{n}}(x,t)|+ e^{-\kappa(T-t)}|D_{x} u_{f_{n}}(x,t)| \leq \dfrac{M}{\sqrt[3]{\kappa}}\|f\|_{\kappa}^{0}.
\]
In the proof of Proposition \ref{approx} we proved, by the application of Ladyzhenskaja et al. \cite[Chapter IV, Theorem 10.1]{Lady} that 
there exists a subsequence $(n_{k}, \; k \in \mathbb{N})$ such that $(u_{f_{n_{k}}}\; k \in \mathbb{N})$, $(D_{x} u_{f_{n_{k}}}\; k \in \mathbb{N})$ are convergent to $u_{f}$ and $D_{x} u_{f}$, uniformly on compact subsets. Therefore, for all $(x,t) \in (0,+\infty) \times [0,T)$ 
\[
e^{-\kappa(T-t)}|u(x,t)|+ e^{-\kappa(T-t)}|D_{x} u(x,t)| \leq \dfrac{M}{\sqrt[3]{\kappa}}\|f\|_{\kappa}^{0}
\]
and consequently
\begin{equation*}
 \| u_{f}\|_{\kappa} \leq \dfrac{M}{\sqrt[3]{\kappa}}\|f\|_{\kappa}^{0}.
\end{equation*}
\end{proof}

For $u \in \mathcal{C}^{1+,0+}_{b,loc}$, we can define the mapping 
\begin{align} \label{t}
\mathcal{T} u (x,t) := \et   \biggl [  &\beta(X_{T\wedge \tau(x,t)}(x,t),T\wedge \tau(x,t)) \\ &+  \int_{t}^{T\wedge \tau(x,t)} H(D_{x} u(X_{s}(x,t),s),u(X_{s}(x,t),s),X_{s}(x,t),s)\; ds \biggr ]. \notag
\end{align}
By defining the function $f$ as
\[
f(x,t):=H(D_{x}u(x,t),u(x,t),x,t),
\]
we are able to apply the previous results to the function $\mathcal{T}u$.

\begin{prop} If conditions A1)--A4) are satisfied, then the operator $\mathcal{T}$ maps $\mathcal{C}^{1+,0+}_{b,loc}$ into $\mathcal{C}^{1+,0+}_{b,loc}$ and there exists a constant $\kappa>1$ such that  mapping  \eqref{t} is a contraction with respect to the norm $\|u\|_{\kappa}$.
\end{prop}

\begin{proof}
We have to prove that the operator $\mathcal{T}$ maps $\mathcal{C}^{1+,0+}_{b,loc}$ to $\mathcal{C}^{1+,0+}_{b,loc}$.
We fix the function $ u \in\mathcal{C}^{1+,0+}_{b,loc}$ and define
\[
w(x,t):= \mathcal{T}u (x,t).
\]
The function $w$ is bounded since $u$, $D_{x} u$, $\beta$ are bounded and the Hamiltonian $H$ satisfies a linear growth condition. Note that 
\[
w=w_{1}+w_{2},
\]
where
\begin{align*}
w_{1}(x,t):&= \et\left[\beta(X_{T\wedge \tau(x,t)}(x,t),T\wedge \tau(x,t))  \right], \\
w_{2}(x,t):&= \et\left[ \int_{t}^{T\wedge \tau(x,t)} H(D_{x} u(X_{s}(x,t),s),u(X_{s}(x,t),s),X_{s}(x,t),s)\; ds \right].
\end{align*}
The function $\beta$ is uniformly Lipschitz continuous with a constant $L_{\beta}$ and thus
\begin{align*}
&|w_{1}(x,t)-w_{1}(\bx,t)| \\ &\leq \ee \left|\beta(X_{T \wedge \tau(x,t)}(x,t),T \wedge \tau(x,t)) - \beta(X_{T \wedge \tau(\bx,t)}(\bx,t),T \wedge \tau(\bx,t)) \right| \\ &\leq L_{\beta} \ee \left| X_{T \wedge \tau(x,t)}(x,t) - X_{T \wedge \tau(\bx,t)}(\bx,t) \right| \\ &\quad \quad +  L_{\beta}\ee \left|(T \wedge \tau(x,t)) - (T \wedge \tau(\bx,t)) \right|. 
\end{align*}

Furthermore, we can assume that $x \geq \bx$. Since \[X_{T \wedge \tau(x,t)}(x,t) \geq X_{T \wedge \tau(\bx,t)}(\bx,t),\] we have
\begin{align*}
\ee \left| X_{T \wedge \tau(x,t)}(x,t) - X_{T \wedge \tau(\bx,t)}(\bx,t) \right|&=
\ee  X_{T \wedge \tau(x,t)}(x,t) -\ee X_{T \wedge \tau(\bx,t)}(\bx,t) \\ &= x-\bar{x} =|x-\bar{x}|.
\end{align*}
By the above and condition {\bf A4)}, we observe that the function $w_1$ is Lipschitz continuous in $x$ uniformly with respect to $t$. As a result, the derivative $D_{x}w_{1}$ is globally bounded.  
Remark \ref{to_norm_estimates} ensures that $D_{x}w_{2}$ is globally bounded. In summary, we get the boundedness of  
$D_{x}w$. Proposition \ref{approx}  and the Feynman-Kac formula (Proposition \ref{Feynman}) guarantee that both $w$ and $D_{x}w$ are H\"older continuous on compact subsets, and consequently $w$ belongs to $C_{b,loc}^{1+,0+}$.

Now, our objective is to prove that $\mathcal{T}$ is a contraction for sufficiently large $\kappa$. 
Fix $u, v \in \mathcal{C}^{1+,0+}_{b,loc}$ and define 
\[
\hat{w}(x,t):=\mathcal{T} u (x,t) - \mathcal{T} v (x,t).
\]  Note that, by Proposition \ref{approx} and the Feynman-Kac formula (Proposition \ref{Feynman}), the function $\hat{w}$ is a classical solution to 
\begin{equation*} 
 \begin{cases}
  D_{t} \hat{w}+ \tfrac{\sigma^{2}(x,t)}{2}D^{2}_{x}\hat{w} + f(x,t) =0,  &\quad (x,t) \in (0,+\infty) \times [0,T), \\
\hat{w}(x,t)=0, &\quad (x,t) \in  \partial_{p} \left((0,+\infty) \times [0,T)\right),
\end{cases}
\end{equation*}
where $f(x,t):=H(D_{x} u, u, x, t) -  H(D_{x} v, v, x, t)$.
By applying Remark \ref{to_norm_estimates}, we obtain that there exists a constant $M>0$ independent of $u$ and $v$, such that 
 \[ \|\hat{w}\|_{\kappa} \leq \dfrac{M}{\sqrt[3]{\kappa}} \| H(D_{x} u, u, x, t) -  H(D_{x} v, v, x, t)\|_{\kappa}^{0} \leq  \dfrac{K M}{\sqrt[3]{\kappa}} \|u-v\|_{\kappa}, \quad \kappa>1.
 \] 
We obtain a contraction for $\kappa > \max\{(KM)^{3},1 \}$.
\end{proof}

\begin{proof}[Proof of Theorem \ref{main}] The rest of the proof is analogous to the proof of Zawisza \cite[Theorem 2.2]{Zawisza1}, but we repeat it for the reader's convenience.
Reasoning is based on a fixed-point type argument for the mapping $\mathcal{T}$.
 We take any $u_{1} \in \mathcal{C}^{1+,0+}_{b,loc}$ and define recursively the sequence 
 \[u_{n+1}:=\mathcal{T} u_{n}, \quad n \in \mathbb{N}. \] There exists $\kappa>0$ such that the mapping $\mathcal{T}$ is a contraction in $\| \cdot \|_{\kappa}$ and this implies that the sequence $u_{n}$ converges to some fixed point $u$. The function $u$ belongs to $\mathcal{C}^{1,0}_{b}$ and we have to prove that $u$ also belongs to class $\mathcal{C}^{1+,0+}_{b,loc}$. First, let us note that the sequences $(u_{n}, n \in \mathbb{N})$ and $(D_{x}u_{n}, n \in \mathbb{N})$  converge in $\|\cdot\|_{\kappa}^{0}$ (for $\kappa$ large enough), so they are bounded uniformly with respect to $n$. We can now exploit (E8) and (E9) of Fleming and Rishel \cite{FlemingRishel} (together with the localization procedure to (E8) described therein) and prove the uniform bound on compact subsets for H\"{o}lder norm of $u_{n}$ and $D_{x} u_{n}$. More precisely, by the Feynman-Kac formula (Proposition \ref{Feynman})  and Proposition \ref{approx}, we know that
 \begin{multline*}
  D_{t}u_{n+1}+ \tfrac{\sigma^{2}(x,t)}{2}D^{2}_{x} u_{n+1}+  H(D_{x}u_{n}(x,t),u_{n}(x,t),x,t)=0,  \\ (x,t) \in (0,+\infty) \times [0,T),
 \end{multline*}
 and
 \[
u_{n+1}(x,t)=\beta(x,t), \quad (x,t) \in  \partial_{p} \left((0,+\infty) \times [0,T)\right).
 \]
Now, consider the set $\mathcal{O}_{k}:=(\frac{1}{k},k)\times (\frac{1}{k},T-\frac{1}{k})$ and a function $\psi_{k} \in \mathcal{C}^{2}$ such that 
$\psi_{k} =1$ on $O_{k}$ and $\psi_{k}=0$ on $O_{k+1}^{c}$ and define $\varphi_{n,k} = \psi_{k} u_{n+1}$. 
We have
\begin{multline} \label{zero}
  D_{t}\varphi_{n,k}+ \tfrac{\sigma^{2}(x,t)}{2}D^{2}_{x} \varphi_{n,k}  = -\psi_{k} H(D_{x}u_{n},u_{n},x,t) \\ + \sigma^{2}(x,t)D_{x} u_{n+1} D_{x} \psi_{k}+\tfrac{\sigma^{2}(x,t)}{2}u_{n+1} D^{2}_{x} \psi_{k} + u_{n+1} D_{t} \psi_{k}
\end{multline}
on $\mathcal{O}_{k+1} $ and $\varphi_{n,k}=0$ on $\partial_{p} \left(\mathcal{O}_{k+1}\right)$.
Note that the right hand side of \eqref{zero} is globally bounded and this means that E8) implies
that there exists a constant $B_{k,\lambda}>0$ such that $\|\varphi_{n,k}\|^{2}_{\lambda,\mathcal{O}_{k+1}} \leq B_{k,\lambda}$, $k,n \in \mathbb{N}$, $\lambda \in (1,+\infty)$. By E9), we get
\[
 \|\varphi_{n,k}\|_{\mathcal{C}^{1+l,l/2}(O_{k+1})}\leq \|\varphi_{n,k}\|^{(2)}_{\lambda,O_{k+1}} \leq B_{k,\lambda}, \quad \lambda > 3,\; l=1-\frac{3}{\lambda},\; k,n \in \mathbb{N}.
\]
Consequently,
\begin{multline*}
\|u_{n+1}\|_{\mathcal{C}^{1+l,l/2}(O_{k})}= \|\varphi_{n,k}\|_{\mathcal{C}^{1+l,l/2}(O_{k})}\leq \|\varphi_{n}\|_{\mathcal{C}^{1+l,l/2}(O_{k+1})}\leq B_{k,\lambda}, \\ \lambda > 3,\; l=1-\frac{3}{\lambda},\; k,n \in \mathbb{N}.
\end{multline*}
Therefore, for all $(x,t)$, $(x,\bar{t}) \in O_{k}$, $t \neq \bar{t}$ we have
\begin{multline*}
\frac{|u_{n+1}(x,t)-u_{n+1}(x,\bar{t})|}{|t-\bar{t}|^{l/2}}+\frac{|D_x u_{n+1}(x,t)-D_x u_{n+1}(x,\bar{t})|}{|t-\bar{t}|^{l/2}} \leq B_{k,\lambda}, \\ \lambda > 3,\; l=1-\frac{3}{\lambda},\; k, n \in \mathbb{N},
\end{multline*}
and for all $(x,t)$, $(\bx,t) \in O_{k}$, $x \neq \bar{x}$
\begin{multline*}
\frac{|u_{n+1}(x,t)-u_{n+1}(\bx,t)|}{|x-\bar{x}|^{l}}+\frac{|D_x u_{n+1}(x,t)-D_x u_{n+1}(\bx,t)|}{|x-\bar{x}|^{l}} \leq B_{k,\lambda}, \\ \lambda > 3,\; l=1-\frac{3}{\lambda},\; k,n \in \mathbb{N}.
\end{multline*}
For all $(x,t)$, $(x,\bar{t}) \in O_{k}$, $t \neq \bar{t}$, after passing to the limit as $n \to +\infty$, we obtain :
 \[
\frac{|u(x,t)-u(x,\bar{t})|}{|t-\bar{t}|^{l/2}}+\frac{|D_x u(x,t)-D_x u(x,\bar{t})|}{|t-\bar{t}|^{l/2}} \leq B_{k,\lambda}, \; \lambda > 3,\; l=1-\frac{3}{\lambda},
\]
and  for all $(x,t)$, $(\bx,t) \in O_{k}$, $x \neq \bar{x}$:
\[
\frac{|u(x,t)-u(\bx,t)|}{|x-\bar{x}|^{l}}+\frac{|D_x u(x,t)-D_x u(\bx,t)|}{|x-\bar{x}|^{l}} \leq B_{k,\lambda}, \; \lambda > 3,\; l=1-\frac{3}{\lambda}.
\]
Therefore, $u$ and $D_{x} u$ are H\"{o}lder continuous on compact subsets of the set $(0,+\infty) \times (0,T)$. In fact, we should require of them to be H\"{o}lder continuous on compact subsets of $(0,+\infty) \times [0,T)$. This, however, can be easily achieved by extended the definition of functions $\sigma$, $H$ and $\beta$, for $t \in [-1,0)$ in the following way:
\[
H(p,u,x,t):=H(p,u,x,0),\quad \beta(x,t):=\beta(x,0),\quad \sigma(x,t):=\sigma(x,0).
\]
Now we can repeat the proof verbatim for the set 
$\mathcal{O}_{k}:=(\frac{1}{k},k) \times (-1,T-\frac{1}{k})$.

Recall that $u$ is a fixed point and therefore,
\begin{align*}
u (x,t) := \et   \biggl [ & \beta(X_{T\wedge \tau(x,t)}(x,t),T\wedge \tau(x,t)) \\ &+  \int_{t}^{T\wedge \tau(x,t)} H(D_{x} u(X_{s}(x,t),s),u(X_{s}(x,t),s),X_{s}(x,t),s)\; ds \biggr ].
\end{align*}
  Proposition \ref{approx} and the Feynman-Kac formula (Proposition \ref{Feynman}) confirm that the fixed point $u$ belongs to the class \[\mathcal{C}^{2,1}((0,+\infty) \times [0,T)) \cap \mathcal{C}([0,+\infty) \times [0,T]) \] and satisfies equation \eqref{firsteq}.
\end{proof}

\begin{rem} \label{unique}  From the proof of Theorem \ref{main} (the contractive property of $\mathcal{T}$), we can infer the uniqueness of the solution to \eqref{second} within the class $\mathcal{C}^{1+,0+}_{b,loc}$. More precisely, if we have two classical solutions $u_{1}  \in \mathcal{C}^{1+,0+}_{b,loc} $ and $u_{2} \in \mathcal{C}^{1+,0+}_{b,loc}$, then $u_{1}\equiv u_{2}$.
\end{rem}

\begin{rem}
Note that we might consider another half--line (not starting at 0) as a domain of equation \eqref{firsteq}. We choose, for example, the interval $[x_{0},+\infty)$, $x_{0} \in \mathbb{R}$, then simple translation $u(x-x_{0})$ will reduce the problem to the domain $[0,+\infty)$, and consequently Theorem \ref{main} will still be valid for a Dirichlet problem in the domain $[x_{0},+\infty)$.
\end{rem}

\begin{rem}
The fixed point iteration method ensures exponential speed of convergence i.e. there exist  constants $C(T)>0$ and $q \in (0,1)$ such that  
\[
\|u-u_{n}\|^{0}+ \|D_{x}u-D_{x}u_{n}\|^{0} \leq C(T)q^{n},
\]
which additionally in a stochastic control context and under certain circumstances  imply exponential convergence for the control sequence $\alpha_{n}$ (cf Kermikulov et al. \cite[Assumption 2.4]{Kerimkulov}).
\end{rem}

\begin{rem}
It should be also noted that in the homogeneous case, by using estimates of the form \[\left|D_{x}\mathbb{E}\left[ f(X_{s}(x,t))\mathbf{1}_{\{\tau(x,t)>s\}} \right]\right| \leq \frac{K_T \sup_{x \in (0,+\infty)} |f(x)|}{\sqrt{s-t}}, \quad s>t{,}\; x>0 \] from  Fornaro et al. \cite{Fornaro} or finally Hieber et al. \cite{Hieber}, it is possible to extend the current results for the equations of the form 
\[  D_{t}u+ \frac{1}{2} \sigma^{2}(x) D^{2}_{x} u + b(x)D_{x} u + H(D_x u,u,x,t) =0, \quad (x,t) \in (0,+\infty) \times [0,T)
\]
and the associated diffusion
\[
dX_{s}=b(X_{s}) \;ds + \sigma(X_{s})\;dW_{s}
\]
with $\sigma$ and $b$ unbounded, but the cost we have to pay is a further regularity assumptions for coefficients $b$ and $\sigma$ and rather long list of additional other conditions. Observe also that when the coefficients $b$ and $\sigma$ are bounded, the term $b(x)D_{x} u$ can be incorporated into the function $H$ and therefore can be treated by our methodology.
\end{rem}

As was mentioned in Introduction, our primary concern is to consider the problem \eqref{second} for specific choice of the function $H$. Namely, we are interested in the control objective
\begin{multline} \label{objective:again}
\mathcal{J}^{\alpha} (x,t) := \ee_{x,t} \biggl[ \int_{t}^{T \wedge \tau^{\alpha}(x,t)} e^{\int_{t}^{s} h(X_{k}^{\alpha},k,\alpha_{k})\; dk } l(X_{s}^{\alpha},s,\alpha_{s})\; ds \\ + e^{\int_{t}^{T \wedge \tau^{\alpha}(x,t)} h(X_{k}^{\alpha},k,\alpha_{k})\; dk }\beta(X^{\alpha}_{T \wedge \tau^{\alpha}(x,t)} ,T \wedge \tau^{\alpha}(x,t))\biggr]
\end{multline}
and the value function $V(x,t):=\sup_{\alpha \in \mathcal{A}_{t}}\mathcal{J}^{\alpha}(x,t)$, where where $\mathcal{A}_{t}$ is an admissible control set and $\{X_{s}^{\alpha}\}_{s \geq t}$ is a strong solution to the initial value problem 
\[
\begin{cases}
dX_{s}^{\alpha}= b(X_{s}^{\alpha},s,\alpha_{s})\;ds + \sigma(X_{s}^{\alpha},s) \; dW_{s}{,} \\ 
X_{s}^{\alpha} = x.
\end{cases}
\] 
To determine an optimal strategy $\hat{\alpha}$ one should solve an HJB equation of the form
\begin{equation}  \label{equationL}  D_{t}u+ \tfrac{\sigma^{2}(x,t)}{2} D^{2}_{x} u  + H(D_{x}u,u,x,t) =0,  \quad (x,t) \in (0,+\infty) \times [0,T)
\end{equation}
with the boundary condition 
\[u(x,t)=\beta(x,t), (x,t) \in \partial_{p} \left((0,+\infty) \times [0,T) \right)\] and 
\begin{equation} \label{Hmax}
H(p,u,x,t):=\max_{\alpha \in A}  \left(b(x,t,\alpha)p +  h(x,t,\alpha) u + l(x,t,\alpha)\right).
\end{equation}
To ensure that all the conditions {\bf A1)}--{\bf A4)} are fulfilled, we propose the following assumptions that are linked to problem formulation \eqref{equationL} with $H$ given by formula \eqref{Hmax}.

 \begin{itemize}
\item[{\bf B1)}]
The coefficient $\sigma$ is  bounded, bounded away from zero and Lipschitz continuous in $(x,t)$,i.e., there exists a constant $L_{\sigma}>0$ such that for all $x,\bar{x} \in \mathbb{R}$, $t,\bar{t} \in [0,T]$
\[
|\sigma(x,t)-\sigma(\bar{x},\bar{t})|\leq L_{\sigma} \left(|x-\bar{x}|+ |t-\bar{t}| \right).
\]

\item[{\bf B2)}]  The function $\beta$ is bounded and Lipschitz continuous, i.e., there exists a constant $L_{\beta}>0$ such that for all $(x,t), (\bar{x},\bar{t}) \in \partial_{p} \left((0,+\infty) \times [0,T) \right)$ 
 \[
 |\beta(x,t)-\beta(\bar{x},\bar{t})|\leq L_{\beta}\left(|x-\bar{x}|+ |t-\bar{t}| \right).
 \]

\item[{\bf B3)}] The functions $l$, $h$, $b$  are continuous and bounded and there exists a constant $L>0$ such  that for all $\zeta=l,h,b$ and for all $\alpha \in A$, $(x,t),(\bar{x},\bar{t}) \in \mathbb{R} \times [0,T]$
\begin{equation*} 
|\zeta(x,t,\alpha)-\zeta(\bar{x},\bar{t},\alpha)|\leq L( |x- \bar{x}| + |t-\bar{t}|). \\  
\end{equation*}

\item[{\bf B4)}] There exists a constant $L>0$ such that for all $(x,t),(\bar{x},t) \in [0,+\infty) \times [0,T]$
 \begin{equation*}
\left| \ee (\tau(x,t) \wedge T) - \ee (\tau(\bar{x},t) \wedge T) \right | \leq L|x-\bar{x}|,
\end{equation*}
where
\[ \tau (x, t): = \inf \{s \geq t | \; X_{s}(x,t)\leq 0\}
\]
and 
\[
dX_{s}=\sigma(X_{s},s)\; dW_{s}.
\]
\end{itemize}

Note that {\bf B1)} mimics {\bf A1)}, {\bf B2)} mimics {\bf A2)}, {\bf B4)} mimics {\bf A4)}, and  it is further shown that {\bf B3)} imply {\bf A3)} for a specific $H$ choice. 
Now we can give an immediate consequence of Theorem \ref{main}. 
\begin{prop} \label{forH}
Assume that all conditions B1)--B4) are satisfied. Then there exists a classical solution $u \in \mathcal{C}^{2,1}((0,+\infty) \times [0,T)) \cap \mathcal{C}([0,+\infty) \times [0,T])$ to \eqref{equationL}, which is bounded together with $D_{x} u$. 
\end{prop}
\begin{proof}
It is sufficient to prove that the condition {\bf B3)} implies {\bf A3)}, when 
\[
H(p,u,x,t):=\max_{\alpha \in A}h(\alpha,p,u,x,t)
\]
and
\[ h(\alpha,p,u,x,t):= b(x,t,\alpha)p +  h(x,t,\alpha) u + l(x,t,\alpha).
\]
This implication follows from the fact that the functions $b$, $h$, $l$ are bounded and the following inequality holds
\begin{equation*}
|\max_{\alpha \in A}h(\alpha,p,u,x,t) - \max_{\alpha \in A}h(\alpha,\bar{p},\bar{u},\bar{x},\bar{t})|  \leq \max_{\alpha \in A}|h(\alpha,p,u,x,t) - h(\alpha,\bar{p},\bar{u},\bar{x},\bar{t})|.
\end{equation*}

\end{proof}
To complete the task of finding the solution to the control problem  \[V(x,t)=\sup_{\alpha \in \mathcal{A}} \mathcal{J}^{\alpha}(x,t)\] it is now sufficient to find any Borel measurable selector and apply the standard verification result. The existence of a Borel selector is guaranteed by the compactness of the set $A$ and continuity of the function 
\[h(\alpha,p,u,x,t):= b(x,t,\alpha)p +  h(x,t,\alpha) u + l(x,t,\alpha).\] For this purpose, the Kuratowski–Ryll-Nardzewski result can be applied (see Wagner \cite[Theorem 3.1]{Wagner}).

\section{Lipschitz continuity for expected terminated hitting time}

Now, we provide sufficient conditions under which point A4) is satisfied. We have the following proposition.

\begin{prop} \label{condition:t} Suppose that $\sigma \in C^{1}(\mathbb{R})$ is bounded and bounded away from zero, its derivative $D_{x} \sigma$ is  bounded and locally Lipschitz continuous. Then the stopping time
 $\tau(x,t):=\{s\geq t|\; X_{s}(x,t) \leq 0\}$, where 
\[
dX_{s}=\sigma(X_{s})\;dW_{s}, \quad X_{t}=x,
\]
 satisfies condition \eqref{property:L}.
\end{prop}
\begin{proof}
The proof consists of three steps.
In the first step, we establish the assertion for a \emph{pure} Brownian motion, i.e., with no drift and constant volatility $\sigma \equiv 1$.
In the second step, we use this property together with Theorem~\ref{main} to prove the assertion for a diffusion process with drift. In the third step, we apply a change-of-variable technique to transfer the results from the diffusion \( dX_s = b(X_s)\,ds + dW_s \) to the case \( dX_s = \sigma(X_s)\,dW_s \).

{\bf Step I}
First,
we consider a trivial dynamics of the form
\[
dX_{s}=\; dW_{s}.
\] 
Note that  
\begin{align*}
\ee (\tau(x,t) \wedge T) &= \int_{t}^{T} P(\tau(x,t) >s)\;ds = \int_{t}^{T}[1-P(\tau(x,t) \leq s)]\;ds \\&=\int_{t}^{T}\left[1-2P(W_{s}-W_{t}>-x)\right]\;ds,\quad x \geq 0,
\end{align*}
where the last equality is guaranteed by the reflection principle.
Therefore, for $x>\bar{x} \geq 0$ we have
\begin{multline*}
\left|  \ee (\tau(x,t) \wedge T) - \ee (\tau(\bar{x},t)  \wedge T)\right| = \int_{t}^{T} \int_{\bx}^{x}\dfrac{1}{2\pi\sqrt{s-t}} e^{-\dfrac{z^{2}}{2(s-t)}} dz ds \\ \leq |x-\bar{x}| \int_{t}^{T} \dfrac{1}{2\pi\sqrt{s-t}}ds  \leq \dfrac{\sqrt{T}}{\sqrt{\pi}} |x-\bar{x}|. 
\end{multline*}

{\bf Step II}
Consider now a SDE of the form
\begin{equation*} \label{with:drift}
dX_{s}=b(X_{s})\;ds + dW_{s},
\end{equation*}
where the function $b$ is Lipschitz continuous on compact subsets of $\mathbb{R}$ and bounded. Step I ensures that we may apply Theorem \ref{main} to  prove that the equation 
\begin{equation*} 
 \begin{cases}
  D_{t}u+ \frac{1}{2}D^{2}_{x}u + b(x)D_{x} u +1=0,  &\quad (x,t) \in (0,+\infty) \times [0,T), \\
u(x,t)=0,  &\quad (x,t) \in \partial_{p} \left((0,+\infty) \times [0,T)\right)
\end{cases}
\end{equation*}
admits a bounded classical solution $u$ with a bounded derivative $D_{x} u$. The   Feynman--Kac representation (Proposition \ref{Feynman}) ensures that $u(x,t)= \ee (\tau(x,t) \wedge T)  -t$. 

{\bf Step III} (The classical Lamperti transform) Let finally  $\sigma$ be a function satisfying the assumptions of the current proposition. Consider the dynamics 
\[
dY_{s}=-\frac{1}{2} D_{x} \sigma(\zeta^{-1}(Y_{s}))\; ds +\;dW_{s},
\]
where
\[
\zeta(x):=\int_{0}^{x}\frac{1}{\sigma(z)}\; dz.
\]
Without loss of generality, we can assume that $\sigma$ is positive. Note that, $\zeta \in \mathcal{C}^{2}((0,+\infty))$, $\zeta$ is increasing  and $\zeta(x)=0$ if and only if $x=0$.
By the It\^{o} formula, we get
\[
d\zeta^{-1}(Y_{s})= D_{x}\zeta^{-1}(Y_{s})\left[-\frac{1}{2} D_{x} \sigma(\zeta^{-1}(Y_{s}))\; ds +\;dW_{s}\right] + \frac{1}{2} D_{x}^{2} \zeta^{-1}(Y_{s})\; ds.
\]
Note that $D_{x}\zeta^{-1}(x)=\sigma(\zeta^{-1}(x))$ and $D_{x}^{2} \zeta^{-1}(x)=\sigma(\zeta^{-1}(x)) D_{x} \sigma(\zeta^{-1}(x)) $, and thus
\[
d\zeta^{-1}(Y_{s})= \sigma(\zeta^{-1}(Y_{s})) \; dW_{s},
\]
 which implies that $\{X_{s}(x,t)\}_{s \geq t}=\{\zeta^{-1}(Y_{s}(\zeta(x),t))\}_{s \geqslant t}$ is a unique strong solution to 
\[
dX_{s}=\sigma(X_{s})\; dW_{s}.
\]
Furthermore,
\begin{align*}
\{s \geq t|\; Y_{s}(\zeta(x),t) \leq 0 \}&= \{s \geq t|\; \zeta^{-1}(Y_{s}(\zeta(x),t)) \leq 0 \}\\ &=\{s \geq t |\;X_{s}(x,t)\leq 0 \}.
\end{align*}
By Step II, condition  \eqref{property:L} is satisfied for the process $\{Y_{s}(\zeta(x),s)\}_{s\geq t}$ and using the fact that $\zeta $ is a Lipschitz continuous function, we get the same for the process $X$.

\end{proof}

\begin{prop} \label{to:gamma} In addition to assumptions of Proposition \ref{condition:t} suppose that the function $\gamma:[0,T] \to \mathbb{R}$ is continuous and bounded away from zero. Then the stopping time $\tau(x,t):=\inf \{ s \geq t|\; X_{s}(x,t)\leq 0\}$, where 
\[
dX_{s}=\sigma(X_{s}) \gamma(s) \; dW_{s}, \quad X_{t}=x,
\]
satisfies condition \eqref{property:L}.
\end{prop}

\begin{proof}
It is sufficient to consider first the dynamics
\[
dX_{s}=\sigma(X_{s}) \; dW_{s}
\]
and the problem
\[
\begin{cases}
  u_{t}+ \frac{1}{2}\sigma^{2}(x) D^{2}_{x}u +\frac{1}{(\gamma(t))^{2}}=0,  &\quad (x,t) \in (0,+\infty)\times [T_0,T), \\
u(x,t)=0,  &\quad (x,t) \in \partial_{p} \left((0,+\infty) \times [T_0,T)\right)
\end{cases}
\]
with $T_{0}:=T-\int_{0}^{T}(\gamma(s))^{2} \;ds$. Let $u$ denote the unique classical solution to the above equation (it exists according to Theorem \ref{main}) and define a new function of the form
\[
v(x,t)=u\left(x,T-\int_{t}^{T}(\gamma(s))^{2}\;ds\right).
\]
The function $v$ satisfies the equation
\[
\begin{cases}
  D_{t} v+ \frac{1}{2}\sigma^{2}(x) (\gamma(t))^{2} D^{2}_{x}v +1=0,  &\quad (x,t) \in (0,+\infty) \times [0,T), \\
v(x,t)=0,  &\quad (x,t) \in \partial_{p} \left((0,+\infty) \times [0,T)\right).
\end{cases}
\]
From the Feynman-Kac representation (Proposition \ref{Feynman}) we know that
\[
v(x,t)=\mathbb{E} (T\wedge  \tau(x,t))-t,
\] 
where 
$\tau(x,t):=\inf \{ s \geq t|\; X_{s}(x,t)\leq 0\}$
and 
\[
dX_{s}=\sigma(X_{s}) \gamma(s) \; dW_{s}, \quad X_{t}=x.
\]

Theorem \ref{main} implies that there exists a constant $L>0$ such that for all $(x,t),(\bx,t) \in [0,+\infty)\times [0,T]$ we have
\begin{align*}
|\mathbb{E}& (T\wedge  \tau(x,t))-\mathbb{E} (T\wedge  \tau(\bx,t))|=|v(x,t)-v(\bx,t)| \\ &=\left|u\left(x,T-\int_{t}^{T}(\gamma(s))^2 \;ds \right)-u\left(\bx,T-\int_{t}^{T}(\gamma(s))^2 \;ds \right)\right|\leq L|x-\bx|.
\end{align*}
\end{proof}

\section*{Appendix}

\subsection{Basic definitions}
\begin{defin}
Let $\mathcal{O} \subset \mathbb{R}^{n}$. We say that a function $f:\mathcal{O} \to \mathbb{R}$ is H\"older continuous on compact subsets if, for every compact set $U \subset O$, there exist a constant 
$L_{U}>0$ and a H\"older exponent $l_{U} \in (0,1]$ such that for all $x,\bx \in U$
\[
|f(x)-f(\bx)| \leq L_{U} \|x-\bx\|^{l_{U}}.
\] 
\end{defin}

\subsection{Norms in spaces of differential functions}
Let $O \times (T_{1},T_{2}) \subset \mathbb{R}^2$ be an open and bounded set. Suppose that $u:O \times (T_{1},T_{2}) \to \mathbb{R}$ is a sufficiently regular function. We use $D_{x}u$, $D_{t}u$, to  denote the  partial derivatives with respect to  $x \in O$ and $t \in (T_{1},T_{2})$ respectively, and $D^{2}_{x} u$ for the second order partial derivative with respect to $x \in O$. For the subset $G \times I \subset O \times (T_{1},T_{2})$, and parameters $l \in (0,1]$, $\lambda \in (1,+\infty)$, we also define
\begin{align*}
&\|u\|_{\mathcal{C}(G \times I)}:=\sup_{(x,t)\in G \times I} |u(x,t)|, \\
&\|u\|_{\lambda,\, G \times I}:= \left[\int_{G} \int_{I} |u(x,t)|^{\lambda}\; dt\; dx \right]^{\frac{1}{\lambda}},\\
 &\begin{aligned} \|u\|_{C^{l,l/2}(G \times I)} 
  :=  \|&u\|_{\mathcal{C}(G \times I)} 
   + \sup_{t \in I,\, x \neq y,\, x,y \in G} 
  \frac{|u(x,t)-u(y,t)|}{|x-y|^{l}} \\ &+ \sup_{x \in G,\, t \neq s, \, s,t \in I} 
  \frac{|u(x,t)-u(x,s)|}{|t-s|^{l/2}},
  \end{aligned} \\
  &\begin{aligned}[t]\|u\|_{\mathcal{C}^{1+l,l/2}(G \times I)} 
  :=\|&u\|_{\mathcal{C}(G \times I)} 
  + \|D_{x}u\|_{\mathcal{C}(G \times I)} 
   +\sup_{x \in G,\, t \neq s,\, s,t \in I} 
  \frac{|u(x,t) -  u(x,s)|}{|t-s|^{l/2}} 
   \\
  &+ \sup_{t \in I,\, x \neq y,\, x,y \in G}\frac{|D_x u(x,t) - D_x u(y,t)|}{|x-y|^{l}} \\ &+ \sup_{x \in G,\, t \neq s,\, s,t \in I} 
  \frac{|D_x u(x,t) - D_x u(x,s)|}{|t-s|^{l/2}},
  \end{aligned} \\  
 &\begin{aligned}[t]  \|u\|_{\mathcal{C}^{2+l,1+l/2}(G \times I)} 
   : =\|&u\|_{\mathcal{C}(G \times I)} 
  +  \|D_{x}u\|_{\mathcal{C}(G \times I)}
  + \|D_{x}^{2}u\|_{\mathcal{C}(G \times I)}
  + \|D_{t}u\|_{\mathcal{C}(G \times I)} \\
  & + \sup_{t \in I,\, x \neq y,\, x,y \in G} 
  \frac{|D^{2}_{x} u(x,t)-D^{2}_{x} u(y,t)|}{|x-y|^l} \\ &+ \sup_{x \in G,\, t \neq s,\, s,t \in I}  
  \frac{|D^{2}_{x} u(x,t)-D^{2}_{x} u(x,s)|}{|t-s|^{l/2}} \\
  & + \sup_{t \in I,\, x \neq y,\, x,y \in G} 
  \frac{|D_t u(x,t)-D_t u(y,t)|}{|x-y|^l} \\ &+ \sup_{x \in G,\, t \neq s,\, s,t \in I}  
  \frac{|D_t u(x,t)-D_t u(x,s)|}{|t-s|^{l/2}},
  \end{aligned}\\
  &\|u\|_{\lambda,\, G \times I}^{(2)}:=\|u\|_{\lambda,\, G \times I}+\|D_x u\|_{\lambda,\, G \times I}+\|D_x^{2} u\|_{\lambda,\, G \times I}+\|D_t u\|_{\lambda,\, G \times I}.
 \end{align*}
Further notations in this direction are introduced on page \pageref{page}, as they require contextual explanation. 
 
\subsection{Probabilistic framework} 
In the  paper we frequently use a stochastic representation of solutions to certain partial differential equations. In such cases we work with stochastic differential equations (either controlled or uncontrolled) of the form 
\[
dX_{s}=b(X_{s},s,\cdot)\; ds + \sigma(X_{s},s,\cdot)\; dW_{s}.
\]
Whenever such an equation appears, it is understood to be considered together with a reference probability space $(\Omega,\mathcal{F},\{\mathcal{F}_{s}\}_{s \geq 0},P)$, where $\{\mathcal{F}_{s}\}_{s \geq 0}$ is the $P$-augmented filtration generated by a one--dimensional Brownian motion $\{W_{s}\}_{s \geq 0}$. The unique strong solution of the initial value problem
\[
dX_{s}=b(X_{s},s,\cdot)\; ds + \sigma(X_{s},s,\cdot)\; dW_{s}, \quad X_{t}=x,
\]
(assuming that the coefficients $b$ and $\sigma$ allow for such a solution) is denoted by $\{X_{s}(x,t)\}_{s \geq t}$. Note that the form of the equation strictly depends on the specific problem under consideration. The symbol $\mathbb{E}[\cdot]$ is used to denote the expected value with respect to the reference probability measure $P$.

\section*{Acknowledgements}
The author would like to thank all referees for a careful and patient reading of the manuscript, as well as for their helpful suggestions.


\begin{thebibliography}{HD}

\bibitem{A} D. Addona, \emph{A semi-linear backward parabolic Cauchy problem with unbounded coefficients of Hamilton Jacobi Bellman type and applications to optimal control}, Appl. Math. Optim. 72 (2014), 1--28.

\bibitem{Addona} D. Addona, L. Angiuli and L. Lorenzi, \emph{Hypercontractivity, supercontractivity,
ultraboundedness and stability in semilinear
problems}, Adv. Nonlinear Anal. 8 (2017), 225--252.


\bibitem{Becherer} D. Becherer and M. Schweizer,  \emph{Classical solutions to reaction--diffusion systems for hedging problems with interacting Ito and point processes}, Ann. Appl. Probab. 15 (2005), 1111--1144.


\bibitem{Bielecki} A. Bielecki, \emph{Une remarque sur la m\'{e}thode de Banach--Cacciopoli--Tikhonov dans la th\'{e}orie des \'{e} quations diff\'{e}rentielles ordinaires}, Bull. Acad. Polon. Sci. 4 (1956), 261--264. 

\bibitem{Buckdahn} R. Buckdahn and T. Y. Nie, \emph{Generalized Hamilton--Jacobi--Bellman equations with Dirichlet boundary
condition and stochastic exit time optimal control problem}, SIAM J. Control Optim. 54 (2016), 
602--631.

\bibitem{Bychowska} A. Bychowska and H. Leszczy\'{n}ski, \emph{Parabolic equations with functional dependence}, Zeit. Anal. Anw. 20 (2001), 115--130.


\bibitem{Calvia} A. Calvia, G. Cappa, F. Gozzi and E. Priola, \emph{HJB Equations and Stochastic Control on Half--Spaces of Hilbert Spaces}, J. Optim. Theory Appl. 198 (2023), 710--744.

\bibitem{Evans} L. Evans, \emph{Partial Differential Equations}, Graduate Studies in Mathematics, American Mathematical Society, 2010.
 
\bibitem{FlemingRishel}  W. H. Fleming and R. W. Rishel,  
\emph{Deterministic  and  stochastic
optimal control},
Springer, New York, 1975.

\bibitem{FS} W. H. Fleming and H. M. Soner, \emph{Controlled Markov Processes and Viscosity Solutions}, 2nd ed. Springer, New York, 2006.

\bibitem{Fornaro} S. Fornaro, G. Metafune and E. Priola, \emph{Gradient estimates for Dirichlet parabolic problems in unbounded domains}, J. Differ. Equ. 205 (2004), 329--353.

\bibitem{Friedman}  A. Friedman, \emph{Partial Differential Equations of Parabolic Type}, Prentice-Hall,
Englewood Cliffs, New Jersey, 1964.

\bibitem{Heath} D. Heath and M. Schweizer,\emph{ Martingales versus PDEs in finance: an equivalence result with examples}, J. Appl. Probab. 37 (2000), 947--957.


\bibitem{Hieber} M. Hieber, L. Lorenzi and A. Rhandi, \emph{Second-order parabolic equations with unbounded coefficients in exterior domains}, Differ. Integral Equ. 20 (2007), 1253--1284.

\bibitem{Ikeda} N. Ikeda and S. Watanabe, \emph{A comparison theorem for solutions of stochastic differential equations and its applications}, Osaka J. Math. 14 (1977), 619--633.


\bibitem{Kerimkulov} B. Kerimkulov, D. \v{S}i\v{s}ka and L. Szpruch, \emph{Exponential Convergence and Stability of Howard's Policy Improvement Algorithm for Controlled Diffusions},  SIAM J. Control Optim. 58 (2020), 1314--1340.

\bibitem{Krylov} N. V. Krylov, \emph{Controlled diffusion processes}, Springer, Berlin, 1980.

\bibitem{Lady}  O. A. Ladyzhenskaja, V. A. Solonnikov and N. N. Uralceva, \emph{Linear and quasilinear equations of parabolic type}, Nauka,
Moscow, 1967 English transl.: American Mathematical Society, Providence, R.I., 1968.

\bibitem{Pardoux} E. Pardoux and S. Peng, \emph{Backward stochastic differential equations and quasilinear parabolic partial differential equations}, in Stochastic Partial Differential Equations and Their Applications: Proceedings of IFIP WG 7/1 International Conference, Springer Berlin, Heidelberg (2005), 200--217.

\bibitem{Rokhlin} D. B. Rokhlin and G. Mironenko, \emph{Regular finite fuel stochastic control problems with exit time}, Math. Method. Oper. Res. 84 (2016), 105--127.

\bibitem{Royer} M. Royer, \emph{BSDEs with a random terminal time driven by a monotone generator and their links
with PDEs}, Stoch. Stoch. Rep. 76 (2004), 281--307.


\bibitem{Rubio1} G. Rubio, \emph{Existence and uniqueness to the Cauchy problem for linear and semilinear parabolic equations with local conditions},  ESAIM: Proc. 31 (2011), 73--100. 

\bibitem{Rubio} G. Rubio, \emph{The Cauchy--Dirichlet Problem for a Class of Linear Parabolic Differential Equations with Unbounded Coefficients in an Unbounded Domain}, Int. J. Stoch. Anal. 2011, 1--35. 
      
\bibitem{Wagner} D. H. Wagner, \emph{Survey of Measurable Selection Theorems: An Update}, In Measure Theory Oberwolfach 1979, edited by Dietrich Kölzow, Lecture Notes in Mathematics 794, Springer, 1980, 176--219.


\bibitem{Zawisza1} D. Zawisza, \emph{Existence results for Isaacs equations with local conditions and related semilinear Cauchy problems}, Ann. Polon. Math. 121 (2018), 175--196. 
 
\bibitem{Zawisza2} D. Zawisza, \emph{Optymalne strategie inwestycyjne wobec ryzyka modelu}, PhD thesis, Jagiellonian University, 2010 (In Polish).

\end{thebibliography}
\end{document}